\documentclass[12pt]{amsart}
\usepackage{txfonts}
\usepackage{bbm}
\usepackage{mathrsfs}
\usepackage{amscd}
\usepackage{amssymb}
\usepackage{amsmath, amsthm, graphicx}
\usepackage[all]{xy}
\usepackage{cite}
\usepackage{tikz}
\usepackage{color}
\usepackage{tikz-cd}
\usepackage{hyperref}
\usepackage[english]{babel}
\usepackage[margin=1.2in]{geometry}

\newtheorem{theorem}{Theorem}[section]
\newtheorem{corollary}[theorem]{Corollary}
\newtheorem{lemma}[theorem]{Lemma}
\newtheorem{example}[theorem]{Example}
\newtheorem{proposition}[theorem]{Proposition}
\newtheorem{definition}[theorem]{Definition}

\newtheorem{remark}[theorem]{Remark}

\begin{document}

\title{Finite generation and holomorphic anomaly equation for equivariant Gromov-Witten invariants of $K_{\mathbb{P}^1\times\mathbb{P}^1}$ }
\author{Xin Wang}
\address{Department of Mathematics\\ Shandong University \\ Jinan, China}

\email{xinwmath@gmail.com}
\begin{abstract}
In this paper,
we prove finite generation property and holomorphic anomaly equation for the equivariant Gromov-Witten  theory of  $K_{\mathbb{P}^1\times\mathbb{P}^1}$. \end{abstract}

\maketitle
\tableofcontents

\allowdisplaybreaks
\section{Introduction}
In the past years, many works has been done about the finite generation and holomorphic anomaly equation  for (non-)compact Calabi-Yau 3-fold and also twisted theory of Calabi-Yau type (c.f. \cite{coates2018gromov}, \cite{fang2019open}, \cite{guo2018structure}, \cite{lhogromov}, \cite{lho2019gromov},
\cite{lho2018stable},
\cite{wang2019quasi}).
Most of the examples studied are about the models of one K{\"a}hler parameter.
In this paper, we study a simple example of non-compact Calabi-Yau threefold with two K{\"a}hler parameter
$K_{\mathbb{P}^1\times\mathbb{P}^1}$. We remark that  this example has been studied in \cite{fang2019open}, \cite{lhogromov} and \cite{lho2019gromov} from different perspectives.

Let $\overline{M}_{g}(\mathbb{P}^1\times\mathbb{P}^1,(d_1,d_2))$ be the moduli space of stable maps to $\mathbb{P}^1\times\mathbb{P}^1$ with genus-$g$, degree $(d_1,d_2)$. Consider the standard linear torus $(\mathbb{C}^*)^4$ action on $\mathbb{P}^1\times\mathbb{P}^1$ with weights $\lambda_0,\lambda_1,\mu_0,\mu_1$. There is a natual lift of the torus action on the moduli space of stable maps  $\overline{M}_{g}(\mathbb{P}^1\times\mathbb{P}^1,(d_1,d_2))$.
The genus-$g$, generating function of  equivariant
Gromov-Witten invariants of $K_{\mathbb{P}^1\times\mathbb{P}^1}$ at point $\tau\in H^*(\mathbb{P}^1\times\mathbb{P}^1;\mathbb{Q})$ is defined by
\[F_g(\tau,q_1,q_2)
=\sum_{d_1,d_2\geq0}\sum_{n\geq0}{q_1}^{d_1}{q_2}^{d_2}\frac{1}{n!}
\int_{\left[\overline{M}_{g,n}(\mathbb{P}^1\times\mathbb{P}^1,(d_1,d_2))\right]^{\text{vir}}_{(\mathbb{C}^*)^4}}\left(\prod_{i=1}^{n}ev_i^*\tau\right) \cdot e_{(\mathbb{C}^*)^4}\left(R^1\pi_*f^*\mathcal{O}(-2,-2)\right)\]
where $\left[\overline{M}_{g,n}(\mathbb{P}^1\times\mathbb{P}^1,(d_1,d_2))\right]^{\text{vir}}_{(\mathbb{C}^*)^4}$ is the equivariant virtual fundamental class (c.f. \cite{li1998virtual}) and $R^1\pi_*f^*\mathcal{O}(-2,-2)$ is the associated obstruction bundle   of Gromov-Witten invariants of $K_{\mathbb{P}^1\times\mathbb{P}^1}$. For simplicity, in the below, we always choose the specialization of the equivariant parameters $\lambda_0=-\lambda_1=\lambda$, $\mu_0=-\mu_1=\mu$. So formally, the genus $g$ generating function $F_g$  satisfies
\[F_g(\tau,q_1,q_2)\in \mathbb{Q}[[\lambda, \mu]][[q_1,q_2]].\]

To state the finite generation property and holomorphic anomaly equation precisely,
firstly, we consider $\left\{\left(M_{\alpha\beta}(q_1,q_2),L_{\alpha\beta}(q_1,q_2)\right): \alpha, \beta\in\{0,1\}\right\}$ which are solutions of the equations
\[\left\{
\begin{aligned}
M^2-\lambda^2& = & q_1(2(M+L))^2 \\
L^2-\mu^2& = & q_2(2(M+L))^2.
\end{aligned}
\right.\]
Using these solutions, it is easy to define a $\mathbb{Q}$ vector space
\[\mathbb{G}_{m,n}=
\left\{\frac{S}{T}: S\in \mathcal{P}_{n}, T\in \mathcal{Q}_{m} \right\}\]
where
\begin{align*}\mathcal{P}_n=&\left\{\text{homogeneous polynomials of} \{L_{\alpha\beta},M_{\alpha\beta},\lambda,\mu:\alpha,\beta\in\{0,1\}\}\ \text{with degree}\ n\right\},
\\
\mathcal{Q}_m=&\left\{\text{monomials of}\ \{\lambda^2 L_{\alpha\beta}+\mu^2 M_{\alpha\beta}:\alpha,\beta\in\{0,1\}\}\ \text{with degree} \ m\right\}.\end{align*}
Then
 we define a finite generated ring
\[\mathbb{G}:=\bigoplus_{k\geq0}\mathbb{G}_{k,3k}.\]
Secondly, we consider the $I$ function of the equivariant Gromov-Witten invariants of $K_{\mathbb{P}^1\times\mathbb{P}^1}$
\begin{align*}
&I(q_1,q_2,z)=z\sum_{d_1,d_2=0}^{\infty}{q_1}^{d_1}{q_2}^{d_2}
\frac{\prod_{k=0}^{2d_1+2d_2-1}(-2H_1-2H_2-kz)}{\prod_{k_1=1}^{d_1}((H_1+k_1z)^2-\lambda^2)
\prod_{k_2=1}^{d_2}((H_2+k_2z)^2-\mu^2)}
=z\left(I_0+\frac{I_1}{z}+\frac{I_2}{z^2}+...\right)
\end{align*}
with $I_0=1$,
$I_1=(2H_1+2H_2)\sum_{d_1,d_2=0,(d_1,d_2)\neq0}^{\infty}
{q_1}^{d_1}{q_2}^{d_2}
\frac{(2d_1+2d_2-1)!}
{(d_1!)^2(d_2!)^2}
.$
 The mirror point is defined as $\tau(q_1,q_2):=\frac{I_1}{I_0}$. From the $I$ function, we define a generator
\[X(q_1,q_2)
:=\left(q_1\frac{d}{dq_1}+q_2\frac{d}{dq_2}\right)\ln\left(1+I_{11}^{2}(q_1,q_2)+I_{11}^{2}(q_2,q_1)\right).\]
where $I_{11}^2(q_1,q_2)=q_1\frac{d}{dq_1}\sum_{d_1,d_2=0,(d_1,d_2)\neq0}^{\infty}
{q_1}^{d_1}{q_2}^{d_2}
\frac{(2d_1+2d_2-1)!}
{(d_1!)^2(d_2!)^2}$.
Thirdly,  we should define another four auxiliary generators \[P_1(q_1,q_2),\ P_2(q_1,q_2),\ P_3(q_1,q_2),\ P_4(q_1,q_2)\]
whose explicit formula are given in Section~\ref{sec:finit-gen}.

Motivated by mirror symmetry, closed formulas and nice properties of $F_g$ are expected to obtain at the mirror point $\tau=\tau(q_1,q_2)$.
Our main theorem is
\begin{theorem}
For $g\geq2$, the genus-$g$ generating function $F_g$ of  equivariant Gromov-Witten invariants of $K_{\mathbb{P}^1\times\mathbb{P}^1}$ satisfies
\begin{itemize}
    \item{
     $F_g$ lies in a finitely generated ring
\[F_g(\tau(q_1,q_2))
\in \mathbb{G}\left[P_1,P_2,P_3,P_4,X\right]\]
Moreover, the degree of $X$ in the polynomial expression of $F_g$ is at most $3g-3$.
}
\item{Holomorphic anomaly equation
\begin{footnotesize}
\begin{align*}
&\frac{d}{dX}F_g(\tau(q_1,q_2))
\\=&-\frac{1}{2\left(\tilde{I}_{22}(q_1,q_2)+\tilde{I}_{22}(q_2,q_1)\right)}
\left(\sum_{g_1+g_2=g}\langle\langle
H_1+H_2\rangle\rangle_{g_1,1}\langle\langle
H_1+H_2\rangle\rangle_{g_2,1}
+\langle\langle
H_1+H_2,H_1+H_2\rangle\rangle_{g-1,2}\right)
\end{align*}
\end{footnotesize}
}
\end{itemize}
where  the definition of $\tilde{I}_{22}(q_1,q_2)$ can be found in Lemma~\ref{lem:Rrecursion}.
\end{theorem}

\begin{remark}
Our computations can also be applied to many other examples of two K{\"a}hler parameters, such as twisted Gromov-Witten  theory over the product of  projective  spaces, local Hirzebruch surfaces and so on. We will deal with these examples in the future.
\end{remark}
 This paper is organised as follows:
In section~\ref{sec:g=0GW} and section~\ref{sec:R}, we focus on the genus-0 twisted equivariant Gromov-Witten theory of  $K_{\mathbb{P}^1\times\mathbb{P}^1}$ and the computation of Givental  $R$ matrix. In section~\ref{sec:finit-gen} and section~\ref{sec:HAE} , we prove the finite generation  property and holomorphic anomaly equation for  Gromov-Witten theory of $K_{\mathbb{P}^1\times\mathbb{P}^1}$. In section~\ref{sec:osc-int}, we discuss the oscillatory integral and  Feymann diagram representation of Givental $R$ matrix.
\\{\bf  Acknowledgements.}
The authors would like to special thank professor Shuai Guo and Felix Janda for numerous discussing Givental theory and Calabi-Yau geometry. The author is partially supported by  NSFC grant 11601279.

\section{Genus-0 equivariant Gromov-Witten theory of $K_{\mathbb{P}^1\times\mathbb{P}^1}$}\label{sec:g=0GW}
\subsection{$(\mathbb{C}^*)^4$ equi-variant theory of $\mathbb{P}^1\times\mathbb{P}^1$ twisted by $\mathcal{O}(-2,-2)$}
First we consider the $(\mathbb{C}^*)^4$ acts on $\mathbb{P}^1\times\mathbb{P}^1$ by
\[(t_0,t_1,s_0,s_1)\cdot ([x_0,x_1],[y_0,y_1]):=([t_0^{-1}x_0,{t_1}^{-1}x_1],[{s_0}^{-1}y_0,{s_1}^{-1}y_1])\]
then the equivariant cohomlogy ring  is
\[H_{(\mathbb{C}^*)^4}^*(\mathbb{P}^1\times\mathbb{P}^1,\mathbb{Q})
=\frac{\mathbb{Q}[\lambda_0,\lambda_1][H_1]}{\langle (H_1-\lambda_0)(H_1-\lambda_1)\rangle}
\otimes_{\mathbb{Q}}\frac{\mathbb{Q}[\mu_0,\mu_1][H_2]}{\langle (H_2-\mu_0)(H_2-\mu_1)\rangle}\]
where $\lambda_i, \mu_j$ are the corresponding equivariant parameter of the torus action, $H_1$ and $H_2$ are the hyperplane classes from the first and second copy of $\mathbb{P}^1$. This torus action can be natually lifted to the canonical bundle $K_{\mathbb{P}^1\times \mathbb{P}^1}=\mathcal{O}(-2,-2)$.  Then we consider the twisted paring on $H_{(\mathbb{C}^*)^4}^*(\mathbb{P}^1\times\mathbb{P}^1;\mathbb{Q})$
\[\langle v_1,v_2\rangle^{tw}=\int_{\mathbb{P}^1\times\mathbb{P}^1}v_1v_2\frac{1}{-2H_1-2H_2}\]
For simplicity, from now on, we take the specialization of the equivariant parameters $\lambda_0=-\lambda_1=\lambda,\ \mu_0=-\mu_1=\mu$, then under the natural basis
$\{1,H_1,H_2,H_1\cdot H_2\}$, the paring matrix  is \begin{align}\label{paring-matrix}
\eta=\frac{1}{2}
\frac{1}{\lambda^2-\mu^2}
\begin{pmatrix}
0&1&-1&0\\
1&0&0&-\lambda^2\\
-1&0&0&\mu^2\\
0&-\lambda^2&\mu^2&0
\end{pmatrix}
\end{align}
For generic equivariant parameters $\lambda$ and $\mu$, the algebra $H^*_{(\mathbb{C}^*)^4}(\mathbb{P}^1\times\mathbb{P}^1;\mathbb{Q})$ is semisimple. It has a canonical basis
\[\mathfrak{e}_{\alpha\beta}=
\frac{1}{4}\left(
1+\frac{H_1}{(-1)^\alpha\lambda}
+\frac{H_2}{(-1)^\beta\mu}+\frac{H_1H_2}{(-1)^{\alpha+\beta}\lambda\mu}\right)\]
which is the idempotent of the semisimple algebra.
\[\sum_{\alpha,\beta}\mathfrak{e}_{\alpha\beta}=1,\quad \mathfrak{e}_{\alpha\beta}\cdot\mathfrak{e}_{\gamma\theta}=\delta_{\alpha\beta}^{\gamma\theta}\mathfrak{e}_{\alpha\beta}.\]
It is easy to show
\[H_1\cdot \mathfrak{e}_{\alpha\beta}=(-1)^\alpha\lambda\mathfrak{e}_{\alpha\beta},\quad H_2\cdot \mathfrak{e}_{\alpha\beta}=(-1)^\beta\mu\mathfrak{e}_{\alpha\beta}.\]

\subsection{Twisted $I$ fucntion}
In this section, we consider a geometric family of elements of the Lagrangian  cone, which is called $I$ function. Using quantum Riemann-Roch theorem (c.f. \cite{coates2009computing} and
\cite{coates2007quantum}), we can obtain the $I$ function of the twisted theory of $\mathcal{O}(-2,-2)$ over $\mathbb{P}^1\times\mathbb{P}^1$ as follows.
\begin{align*}
&I(q_1,q_2,z)=z\sum_{d_1,d_2=0}^{\infty}{q_1}^{d_1}{q_2}^{d_2}
\frac{\prod_{k=0}^{2d_1+2d_2-1}(-2H_1-2H_2-kz)}{\prod_{k_1=1}^{d_1}((H_1+k_1z)^2-\lambda^2)
\prod_{k_2=1}^{d_2}((H_2+k_2z)^2-\mu^2)}
=z\left(I_0+\frac{I_1}{z}+\frac{I_2}{z^2}+...\right)
\end{align*}
with $I_0=1$,
$I_1=(2H_1+2H_2)\sum_{d_1,d_2=0,(d_1,d_2)\neq0}^{\infty}
{q_1}^{d_1}{q_2}^{d_2}
\frac{(2d_1+2d_2-1)!}
{(d_1!)^2(d_2!)^2}
.$
 A very important property of the $I$ function is: it is solutions of the Picard-Fuchs equations
 \begin{align}
 \label{eq:PF}\left\{\begin{aligned}
\left({D_{H_1}}^2-\lambda^2
-q_1\left(2D_{H_1}+2D_{H_2}
\right)\left(2D_{H_1}+2D_{H_2}+z
\right)\right)I=0\\
\left({D_{H_2}}^2-\mu^2
-q_2\left(2D_{H_1}+2D_{H_2}
\right)\left(2D_{H_1}+2D_{H_2}+z
\right)\right)I=0
\end{aligned}\right.\end{align}
where $D_{H_i}=H_i+zq_i\frac{d}{dq_i}$ for $i=1,2$.

\subsection{Quantum differential equation}
The quantum differential equation is of the form $dS(t,q_1,q_2)=dt*_t S(t,q_1,q_2)$. At the mirror point $t=\tau(q_1,q_2)$, the quantum differential equation becomes:
\begin{scriptsize}
\begin{align}
\label{eq:QDE-S}\left\{\begin{aligned}
&D_{H_1}\left(S(\tau(q_1,q_2),q_1,q_2,z)^*(1,H_1,H_2,H_1H_2)\right)=S(\tau(q_1,q_2),q_1,q_2,z)^*\left((H_1+q_1\frac{d\tau(q_1,q_2)}{dq_1})*_{\tau(q_1,q_2)}(1,H_1,H_2,H_1H_2)\right)
\\&D_{H_2}\left(S(\tau(q_1,q_2),q_1,q_2,z)^*(1,H_1,H_2,H_1H_2)\right)=S(\tau(q_1,q_2),q_1,q_2,z)^*\left((H_2+q_2\frac{d\tau(q_1,q_2)}{dq_2})*_{\tau(q_1,q_2)}(1,H_1,H_2,H_1H_2)\right)
\end{aligned}\right.
\end{align}\end{scriptsize}
Assume
\[(H_1+q_1\frac{d\tau(q_1,q_2)}{dq_1})*_{\tau(q_1,q_2)}(1,H_1,H_2,H_1H_2)=(1,H_1,H_2,H_1H_2)A_1\]
and
\[(H_2+q_2\frac{d\tau(q_1,q_2)}{dq_2})*_{\tau(q_1,q_2)}(1,H_1,H_2,H_1H_2)=(1,H_1,H_2,H_1H_2)A_2.\]

Via Birkhoff factorization (c.f. \cite{coates2007quantum}), we obtain
\begin{lemma} The quantum differential matrix
\begin{align*}
&A_1
=\begin{tiny}
\begin{pmatrix}
0&\lambda^2 I_{22a}^{1;\lambda^2}(q_1,q_2)+\mu^2 I_{22a}^{1;\mu^2}(q_1,q_2)&\lambda^2 I_{22a}^{2;\lambda^2}(q_1,q_2)+\mu^2 I_{22a}^{2;\mu^2}(q_1,q_2)&0\\
1+I^2_{11}(q_1,q_2)&0&0&\lambda^2 I_{33a}^{1;\lambda^2}(q_1,q_2)+\mu^2 I_{33a}^{1;\mu^2}(q_1,q_2)\\
I^2_{11}(q_1,q_2)&0&0&\lambda^2 I_{33a}^{2;\lambda^2}(q_1,q_2)+\mu^2 I_{33a}^{2;\mu^2}(q_1,q_2)\\
0&I_{22}^{1}(q_1,q_2)&I_{22}^{2}(q_1,q_2)&0
\end{pmatrix}
\end{tiny}
\end{align*}
and
\begin{align*}
&A_2=
\begin{tiny}\begin{pmatrix}
0&\lambda^2 I_{22a}^{2;\mu^2}(q_2,q_1)+\mu^2 I_{22a}^{2;\lambda^2}(q_2,q_1)&\lambda^2 I_{22a}^{1;\mu^2}(q_2,q_1)+\mu^2 I_{22a}^{1;\lambda^2}(q_2,q_1)&0\\
I^2_{11}(q_2,q_1)&0&0&\lambda^2 I_{33a}^{2;\mu^2}(q_2,q_1)+\mu^2 I_{33a}^{2;\lambda^2}(q_2,q_1)\\
1+I^2_{11}(q_2,q_1)&0&0&\lambda^2  I_{33a}^{1;\mu^2}(q_2,q_1)+\mu^2 I_{33a}^{1;\lambda^2}(q_2,q_1)\\
0&I_{22}^{2}(q_2,q_1)&I_{22}^{1}(q_2,q_1)&0
\end{pmatrix}
\end{tiny}
\end{align*}

\end{lemma}
\begin{proof}
First, we notice that for every $k\geq0$,
\[I_k=I_{k}(q_1,q_2,H_1,H_2,\lambda^2,\mu^2)\] is
a homogenous polynomial of $H_1,H_2,\lambda,\mu$ with degree $k$ and the coefficients are hyper-geometric series of $q_1$ and $q_2$.

It is well known that the differential operators $-zq_1\frac{d}{dq_1}+H_1$ and $-zq_2\frac{d}{dq_2}+H_2$ preserves the tangent space of the Lagrangian cone $T_{J(\tau(q_1,q_2),-z)}\mathcal{L}$.
\begin{align}\label{eq:D1_S*1}
&(zq_1\frac{d}{dq_1}+H_1)S(\tau(q_1,q_2),q_1,q_2,z)^*1
=(zq_1\frac{d}{dq_1}+H_1)\left(I_0+\frac{I_1}{z}+\frac{I_2}{z^2}
+\frac{I_3}{z^3}+...\right)\nonumber
\\=&I^1_{11}(q_1,q_2) S^*H_1+I^2_{11}(q_1,q_2) S^*H_2
\end{align}
where
\[I^1_{11}(q_1,q_2)=I^2_{11}(q_1,q_2)+1
=1+q_1\frac{d}{dq_1}I_1.\]
Similarly,
\begin{align}\label{eq:D2_S*1}
&(zq_2\frac{d}{dq_2}+H_2)S(\tau(q_1,q_2),q_1,q_2,z)^*1
=(zq_2\frac{d}{dq_2}+H_2)\left(I_0+\frac{I_1}{z}+\frac{I_2}{z^2}
+\frac{I_3}{z^3}+...\right)\nonumber
\\=&I^2_{11}(q_2,q_1) S^*H_1+I^1_{11}(q_2,q_1) S^*H_2.
\end{align}

From equations~\eqref{eq:D1_S*1} and \eqref{eq:D2_S*1}, we can get the expression of $S^*H_1$ and $S^*H_2$.
then \[[z^{-k}]S^*H_1,\quad [z^{-k}]S^*H_2\] are both
 homogenous polynomials of $H_1,H_2,\lambda,\mu$ with degree $k+1$ and the coefficients are hyper-geometric series of $q_1$ and $q_2$.
So we have the following
\begin{small}
\begin{align*}
&(zq_1\frac{d}{dq_1}+H_1)S(\tau(q_1,q_2),q_1,q_2,z)^*H_1
=\left(\lambda^2I^{1;\lambda^2}_{22a}(q_1,q_2)+\mu^2I^{1;\mu^2}_{22a}(q_1,q_2)\right) S^*1
+I^1_{22}(q_1,q_2) S^*(H_1H_2)\\
&(zq_2\frac{d}{dq_2}+H_2)S(\tau(q_1,q_2),q_1,q_2,z)^*H_2
=\left(\mu^2I^{1;\lambda^2}_{22a}(q_2,q_1)+\lambda^2I^{1;\mu^2}_{22a}(q_2,q_1)\right) S^*1
+I^1_{22}(q_2,q_1) S^*(H_1H_2)
\end{align*}
\end{small}
From equations~\eqref{eq:D1_S*1} and \eqref{eq:D2_S*1}, we can get the expression of $S^*H_1H_2$,
then \[[z^{-k}]S^*(H_1H_2)\] is a
 homogeneous polynomials of $H_1,H_2,\lambda,\mu$ with degree $k+2$ and the coefficients are hyper-geometric series of $q_1$ and $q_2$. Then taking derivatives, we get
 \begin{scriptsize}
\begin{align*}
&(zq_1\frac{d}{dq_1}+H_1)S(\tau(q_1,q_2),q_1,q_2,z)^*H_1H_2
=\left(\lambda^2I^{1;\lambda^2}_{33a}(q_1,q_2)+\mu^2I^{1;\mu^2}_{33a}(q_1,q_2)\right) S^*H_1+
\left(\lambda^2I^{2;\lambda^2}_{33a}(q_1,q_2)+\mu^2I^{2;\mu^2}_{33a}(q_1,q_2)\right) S^*H_2
\\
&(zq_2\frac{d}{dq_2}+H_2)S(\tau(q_1,q_2),q_1,q_2,z)^*H_1H_2
=\left(\lambda^2I^{2;\mu^2}_{33a}(q_2,q_1)+\mu^2I^{2;\mu^2}_{33a}(q_2,q_1)\right) S^*H_1+
\left(\lambda^2I^{1;\mu^2}_{33a}(q_2,q_1)+\mu^2I^{2;\lambda^2}_{33a}(q_2,q_1)\right) S^*H_2
\end{align*}
\end{scriptsize}

\end{proof}



\subsection{Relations among the entries of matrixes $A_1$ and $A_2$}

From the compatibility of the quantum product and inner product,  we get the following linear  relations
among the entries of $A_1$ and $A_2$.
\begin{lemma}
\label{lem:linear-rel-entr}
\begin{align*}
(i)&\quad I_{22a}^{1;\mu^2}(q_1,q_2)+I_{22a}^{2;\mu^2}(q_1,q_2)=I_{22}^{1}(q_1,q_2)\\
(ii)&\quad I_{22a}^{1;\lambda^2}(q_1,q_2)+I_{22a}^{2;\lambda^2}(q_1,q_2)=I_{22}^{2}(q_1,q_2)
\\(iii)&\quad I_{33a}^{1;\lambda^2}(q_1, q_2)-I_{33a}^{2;\lambda^2}(q_1, q_2)=- \left(1+I_{11}^{2}(q_1, q_2)\right)
\\(iv)&\quad I_{33a}^{1;\mu^2}(q_1, q_2)-I_{33a}^{2;\mu^2}(q_1, q_2)= I_{11}^{2}(q_1, q_2)
\end{align*}

\end{lemma}
\begin{proof}
For any $v,w \in H^*(\mathbb{P}^1\times\mathbb{P}^1;\mathbb{Q})$,
\[\langle q_1\frac{d}{dq_1}\tau(q_1,q_2) *_{\tau(q_1,q_2)}v,w\rangle=\langle v, q_1\frac{d}{dq_1}\tau(q_1,q_2) *_{\tau(q_1,q_2)}w\rangle\]
this is equivalent to  $\eta \cdot A_1$ is a symmetric matrix, then direct computation gives the Lemma~\ref{lem:linear-rel-entr}.
\end{proof}

By the definitions of canonical basis $\left\{e_{\alpha\beta}:\alpha,\beta\in\{0,1\}\right\}$ and canonical coordinates $\left\{u^{\alpha\beta}:\alpha,\beta\in\{0,1\}\right\}$,
\[\left\{\begin{aligned}(H_1+q_1\frac{d\tau(q_1,q_2)}{dq_1})*_{\tau(q_1,q_2)}e_{\alpha\beta}(q_1,q_2)=\left(\mathfrak{h}_{\alpha\beta}^{(1)}+q_1\frac{du^{\alpha\beta}}{dq_1}\right)e_{\alpha\beta}(q_1,q_2)\\
(H_2+q_2\frac{d\tau(q_1,q_2)}{dq_2})*_{\tau(q_1,q_2)}e_{\alpha\beta}(q_1,q_2)=\left(\mathfrak{h}_{\alpha\beta}^{(2)}+q_2\frac{du^{\alpha\beta}}{dq_2}\right)e_{\alpha\beta}(q_1,q_2)\end{aligned}
\right.\]
where $h_{\alpha\beta}^{(1)}=(-1)^\alpha\lambda,\, h_{\alpha\beta}^{(2)}=(-1)^\beta\mu.$
Now two define functions \[M_{\alpha\beta}:=\mathfrak{h}_{\alpha\beta}^{(1)}+q_1\frac{du^{\alpha\beta}}{dq_1}
,\quad L_{\alpha\beta}=\mathfrak{h}_{\alpha\beta}^{(2)}+q_2\frac{du^{\alpha\beta}}{dq_2}\]
then $\{M_{\alpha\beta}\}, \{L_{\alpha\beta}\}$ are eigenvalues of the quantum matrix $A_1$ and $A_2$ respectively. Since the matrices $A_1$ and $A_2$ can be diagonalized  simultaneously, we get
\begin{align}\label{eq:com}
A_1A_2=A_2A_1
\end{align}
The commutation relation \eqref{eq:com} gives
\begin{lemma}
\label{lem:I33a}\begin{tiny}
\begin{align*}
I_{33a}^{2;\lambda^2}(q_1, q_2)
=&\frac{1}{I_{22}^{1}(q_1, q_2)I_{22}^{1}(q_2, q_1)-I_{22}^{2}(q_1, q_2)I_{22}^{2}(q_2, q_1)}\Bigg(\bigg(\left(2I_{11}^{2}(q_2, q_1)+1\right)I_{22}^{1}(q_1, q_2)+\left(I_{22}^{1}(q_2, q_1)-I_{22}^{2}(q_2, q_1)-I_{22a}^{1;\mu^2}(q_2, q_1)\right)I_{11}^{2}(q_1, q_2)
\\&\hspace{150pt}-I_{11}^{2}(q_2, q_1)I_{22a}^{1;\lambda^2}(q_1,q_2)+I_{22}^{1}(q_2, q_1)-I_{22}^{2}(q_2, q_1)-I_{22a}^{1;\mu^2}(q_2, q_1)\bigg)I_{22}^{2}(q_1, q_2)
\\&\hspace{150pt}-\left(I_{11}^{2}(q_1, q_2)I_{22a}^{1;\mu^2}(q_2, q_1)+I_{22a}^{1;\lambda^2}(q_1, q_2)\left(1+I_{11}^{2}(q_2, q_1)\right)\right)I_{22}^{1}(q_1, q_2)\Bigg)\\
I_{33a}^{2;\mu^2}(q_1, q_2)
=&\frac{1}{I_{22}^{1}(q_1, q_2)I_{22}^{1}(q_2, q_1)-I_{22}^{2}(q_1, q_2)I_{22}^{2}(q_2, q_1)}
\Bigg((1+I_{11}^{2}(q_2, q_1))I_{22}^{1}(q_1, q_2)^2-\bigg((I_{11}^{2}(q_2, q_1)+1)I_{22}^{2}(q_1, q_2)+I_{11}^{2}(q_1, q_2)I_{22a}^{1;\lambda^2}(q_2, q_1)
\\&\hspace{150pt}+I_{11}^{2}(q_2, q_1)I_{22a}^{1;\mu^2}(q_1, q_2)+I_{22a}^{1;\mu^2}(q_1, q_2)\bigg)I_{22}^{1}(q_1, q_2)+I_{22}^{2}(q_1, q_2)\bigg(-I_{11}^{2}(q_2, q_1)I_{22a}^{1;\mu^2}(q_1, q_2)
\\&\hspace{150pt}-\left(I_{11}^{2}(q_1, q_2)+1\right)I_{22a}^{1;\lambda^2}(q_2, q_1)+I_{22}^{2}(q_2, q_1)\left(2I_{11}^{2}(q_1, q_2)+1\right)\bigg)\Bigg)
\end{align*}
\end{tiny}
\end{lemma}
\begin{proof}
Let $C=A_1A_2-A_2A_1$, then solving the equation
\[[\mu^2]C_{2,2}=[\mu^2]C_{3,3}=0\] we can obtain the above two identities in Lemma~\ref{lem:I33a}.
\end{proof}
Together with Lemma~\ref{lem:linear-rel-entr}, all the entries $\{I_{33a}^{\bullet;\bullet}\}$  can be written as rational functions in terms of other entries of $A_1$ and $A_2$. For other entries of the matrices $A_1$ and $A_2$, we have many relations given in the following lemma.
\begin{lemma}
\label{lem:YY-rel}\begin{align*}
&(i)\quad
\left(1+I_{11}^{2}(q_1,q_2)\right)I_{22}^{1}(q_1,q_2)+I_{11}^{2}(q_1.q_2)I_{22}^{2}(q_1,q_2)
\nonumber\\&\hspace{20pt}=
4q_1\left(I_{22}^{1}(q_1,q_2)+I_{22}^{2}(q_2,q_1)+I_{22}^{1}(q_2,q_1)+I_{22}^{2}(q_1,q_2)\right)
\left(1+I_{11}^{2}(q_1,q_2)+I_{11}^{2}(q_2,q_1)\right)
\\&(ii)\quad \frac{d}{d q_1}I_{11}^{2}(q_1, q_2)
=\frac{d}{d q_2}I_{11}^{2}(q_2, q_1)
\\&(iii)\quad 4q_1\frac{d}{dq_1}I_{11}^{2}(q_2, q_1)+4q_2\frac{d}{dq_2}I_{11}^{2}(q_1, q_2)+2I_{11}^{2}(q_1, q_2)+2I_{11}^{2}(q_2, q_1)+2
=(1-4q_1-4q_2)
\frac{d}{dq_2}I_{11}^{2}(q_2, q_1)
\\&(iv)\quad I_{22a}^{1;\lambda^2}(q_1, q_2) =1-I_{11}^{2}(q_1, q_2) I_{22}^{2}(q_1, q_2)
+4q_1\left(1+I_{11}^{2}(q_1, q_2)+I_{11}^{2}(q_2, q_1)\right)
\left(I_{22}^{1}(q_2, q_1)+I_{22}^{2}(q_1, q_2)\right)
\\&
(v)\quad I_{22a}^{1;\mu^2}(q_1, q_2) = -I_{11}^{2}(q_1, q_2)I_{22}^{1}(q_1, q_2)
+4q_1 \left(I_{22}^{1}(q_1, q_2)+I_{22}^{2}(q_2, q_1)\right)\left(1+I_{11}^{2}(q_2, q_1)+I_{11}^{2}(q_1, q_2)\right)\\&
(vi)\quad I_{22}^{2}(q_1, q_2)
=-\frac{1}{(1+I_{11}^{2}(q_1, q_2)+I_{11}^{2}(q_2, q_1))(4q_1-4q_2-1)}
+\frac{-4q_1+4q_2-1}{4q_1-4q_2-1}I_{22}^{1}(q_2, q_1)
\end{align*}

\end{lemma}
\begin{proof}
First we write the first Picard-Fuchs equation in \eqref{eq:PF}
as matrix form
\begin{align*}
0=&\left({D_{H_1}}^2-\lambda^2
-q_1\left(2D_{H_1}+2D_{H_2}
\right)\left(2D_{H_1}+2D_{H_2}+z
\right)\right)S(\tau(q_1,q_2),q_1,q_2)^*1
\\=&S(\tau(q_1,q_2),q_1,q_2)^*(1,H_1,H_2,H_1H_2)\cdot B
\end{align*}
where $B$ is a $4\times1$ column matrix:
\begin{align*}
B=&\left(A_1+zq_1\frac{d}{dq_1}\right)\begin{pmatrix}
0\\
1+I_{11}^2(q_1,q_2)\\
I_{11}^2(q_1,q_2)\\
0
\end{pmatrix}
-\lambda^2\begin{pmatrix}
1\\0\\0\\0
\end{pmatrix}
\\&-2q_1\left(2(A_1+A_2)+z+2\left(q_1\frac{d}{dq_1}+q_2\frac{d}{dq_2}\right)\right)\begin{pmatrix}
0\\
1+I_{11}^2(q_1,q_2)+I_{11}^2(q_2,q_1)\\
1+I_{11}^2(q_1,q_2)+I_{11}^2(q_2,q_1)\\
0
\end{pmatrix}.
\end{align*}

Then the first equation $(i)$ just follows from the equation $B_{4,1}=0$.
The equations $(ii)$ and $(iii)$ follow from  $B_{2,1}=0$ and $B_{3,1}=0$. 
Together with the relations $(i)$ and $(ii)$ in Lemma~\ref{lem:linear-rel-entr}, the two equations $[\lambda^2]B_{1,1}=0$  and  $[\mu^2]B_{1,1}=0$ imply the equations $(iv), (v)$ in Lemma~\ref{lem:YY-rel}.
 To prove $(vi)$, we need the equation below, which are derived from $[\lambda^2]C_{1,1}=0$
 \[-I_{22}^{1}(q_2, q_1)I_{11}^{2}(q_1, q_2)+I_{22}^{2}(q_1, q_2)I_{11}^{2}(q_2, q_1)-I_{22}^{1}(q_2, q_1)+I_{22}^{2}(q_1, q_2) = I_{22a}^{1;\lambda^2}(q_1, q_2)-I_{22a}^{1;\mu^2}(q_2, q_1).\]
 Then combing with equations $(iv)$ and $(v)$ in Lemma~\ref{lem:YY-rel}, we can obtain $(vi)$.
\end{proof}

\section{$R$ matrix computations}\label{sec:R}

\subsection{QDE for $R$ matrix}
Following from the quantum differential equations \eqref{eq:QDE-S} and the relation between $S$ and $R$ matrix (c.f. \cite{givental2001semisimple}), we get the quantum differential equations satisfied by  the vector $(R_{1|\bar{\alpha\beta}},R_{H_1|\bar{\alpha\beta}},R_{H_2|\bar{\alpha\beta}},R_{H_1H_2|\bar{\alpha\beta}})$.
\begin{align}\label{eq:qde-R}
\left\{\begin{aligned}
\left(zq_1\frac{d}{dq_1}+M_{\alpha\beta}(q_1,q_2)\right)
(R_{1|\bar{\alpha\beta}},R_{H_1|\bar{\alpha\beta}},R_{H_2|\bar{\alpha\beta}},R_{H_1H_2|\bar{\alpha\beta}})
=(R_{1|\bar{\alpha\beta}},R_{H_1|\bar{\alpha\beta}},R_{H_2|\bar{\alpha\beta}},R_{H_1H_2|\bar{\alpha\beta}})A_1\\
\left(zq_2\frac{d}{dq_2}+L_{\alpha\beta}(q_1,q_2)\right)
(R_{1|\bar{\alpha\beta}},R_{H_1|\bar{\alpha\beta}},R_{H_2|\bar{\alpha\beta}},R_{H_1H_2|\bar{\alpha\beta}})
=(R_{1|\bar{\alpha\beta}},R_{H_1|\bar{\alpha\beta}},R_{H_2|\bar{\alpha\beta}},R_{H_1H_2|\bar{\alpha\beta}})A_2
\end{aligned}\right.\end{align}
Direct consequences of the quantum differential equations for the $R$ matrix are the following recursion relations of the $R$ matrix.
\begin{lemma}
\label{lem:Rrecursion}
The $R$ matrix satisfies the following recursion
\begin{tiny}
\begin{align*}
(i) \quad{R(z)_{H_1}}^{\alpha\beta}
=&\frac{(1+I_{11}^{2}(q_2,q_1))M_{\alpha\beta}-I_{11}^{2}(q_1,q_2)L_{\alpha\beta}}{1+\bar{I}_{11}^{2}(q_1,q_2)}{R(z)_1}^{\alpha\beta}
\\&+z\frac{1}{1+\bar{I}_{11}^{2}(q_1,q_2)}\Bigg((1+I_{11}^{2}(q_2,q_1))q_1\frac{d}{dq_1}{R(z)_1}^{\alpha\beta}-I_{11}^{2}(q_1,q_2)q_2\frac{d}{dq_2}{R(z)_1}^{\alpha\beta}
\\&\hspace{150pt}+(1+I_{11}^{2}(q_2,q_1))\frac{q_1\frac{d}{dq_1}\|e_{\alpha\beta}\|}{\|e_{\alpha\beta}\|}{R(z)_1}^{\alpha\beta}
-I_{11}^{2}(q_1,q_2)\frac{q_2\frac{d}{dq_2}\|e_{\alpha\beta}\|}{\|e_{\alpha\beta}\|}{R(z)_1}^{\alpha\beta}\Bigg)\\
(ii) \quad {R(z)_{H_2}}^{\alpha\beta}
=&\frac{(1+I_{11}^{2}(q_1,q_2))L_{\alpha\beta}-I_{11}^{2}(q_2,q_1)M_{\alpha\beta}}{1+\bar{I}_{11}^{2}(q_1,q_2)}{R(z)_1}^{\alpha\beta}
\\&+z\frac{1}{1+\bar{I}_{11}^{2}(q_1,q_2)}\Bigg((1+I_{11}^{2}(q_1,q_2))q_2\frac{d}{dq_2}{R(z)_1}^{\alpha\beta}-I_{11}^{2}(q_2,q_1)q_1\frac{d}{dq_1}{R(z)_1}^{\alpha\beta}
\\&\hspace{150pt}+(1+I_{11}^{2}(q_1,q_2))\frac{q_2\frac{d}{dq_2}\|e_{\alpha\beta}\|}{\|e_{\alpha\beta}\|}{R(z)_1}^{\alpha\beta}
-I_{11}^{2}(q_2,q_1)\frac{q_1\frac{d}{dq_1}\|e_{\alpha\beta}\|}{\|e_{\alpha\beta}\|}{R(z)_1}^{\alpha\beta}\Bigg)
\\(iii)\quad
{R(z)_{H_1H_2}}^{\alpha\beta}
=&z\frac{1}{\tilde{I}_{22}(q_1,q_2)+\tilde{I}_{22}(q_2,q_1)}\left(\frac{q_2\frac{d}{dq_2}\|e_{\alpha\beta}\|}{\|e_{\alpha\beta}\|}+\frac{q_1\frac{d}{dq_1}\|e_{\alpha\beta}\|}{\|e_{\alpha\beta}\|}\right)\Bigg({R(z)_{H_1}}^{\alpha\beta}+{R(z)_{H_2}}^{\alpha\beta}
\Bigg)
\\&-z\frac{1}{\tilde{I}_{22}(q_1,q_2)+\tilde{I}_{22}(q_2,q_1)}\cdot X(q_1,q_2)
\cdot \Bigg({R(z)_{H_1}}^{\alpha\beta}+{R(z)_{H_2}}^{\alpha\beta}
\Bigg)
\\&+z\frac{1}{(1+\bar{I}_{11}^2)(\tilde{I}_{22}(q_1,q_2)+\tilde{I}_{22}(q_2,q_1))}
\left(q_1\frac{d}{dq_1}+q_2\frac{d}{dq_2}\right)\left((L_{\alpha\beta}+M_{\alpha\beta}){R(z)_1}^{\alpha\beta}\right)
\\&+z^2\frac{1}{(1+\bar{I}_{11}^2)(\tilde{I}_{22}(q_1,q_2)+\tilde{I}_{22}(q_2,q_1))}
\left(q_1\frac{d}{dq_1}+q_2\frac{d}{dq_2}\right)\Bigg(q_2\frac{d}{dq_2}{R(z)_1}^{\alpha\beta}+q_1\frac{d}{dq_1}{R(z)_1}^{\alpha\beta}
+\frac{q_2\frac{d}{dq_2}\|e_{\alpha\beta}\|}{\|e_{\alpha\beta}\|}{R(z)_1}^{\alpha\beta}
+\frac{q_1\frac{d}{dq_1}\|e_{\alpha\beta}\|}{\|e_{\alpha\beta}\|}{R(z)_1}^{\alpha\beta}\Bigg)
\\&+\frac{1}{\tilde{I}_{22}(q_1,q_2)+\tilde{I}_{22}(q_2,q_1)}\left(L_{\alpha\beta}+M_{\alpha\beta}\right)
\left({R(z)_{H_1}}^{\alpha\beta}+{R(z)_{H_2}}^{\alpha\beta}\right)
\\&-\left(\mu^2\frac{\tilde{I}_{22}(q_1,q_2)}{\tilde{I}_{22}(q_1,q_2)+\tilde{I}_{22}(q_2,q_1)} +\lambda^2 \frac{\tilde{I}_{22}(q_2,q_1)}{\tilde{I}_{22}(q_1,q_2)+\tilde{I}_{22}(q_2,q_1)}
\right){R(z)_{1}}^{\alpha\beta}
\end{align*}\end{tiny}
where we define
\begin{align*}X(q_1,q_2)
=\left(q_1\frac{d}{dq_1}+q_2\frac{d}{dq_2}\right)\ln\left(1+I_{11}^{2}(q_1,q_2)+I_{11}^{2}(q_2,q_1)\right),
\\
\bar{I}_{11}^{2}(q_1,q_2)=I_{11}^{2}(q_1,q_2)+I_{11}^{2}(q_2,q_1),\quad \tilde{I}_{22}(q_1,q_2)=I_{22}^{1}(q_1,q_2)+I_{22}^{2}(q_2,q_1).
\end{align*}

\end{lemma}
\begin{proof}
The first column of quantum differential equations~\eqref{eq:qde-R}  give the recursion relations
\[(1+I_{11}^{2}(q_1,q_2))R(z)_{H_1|\bar{\alpha\beta}}+I_{11}^{2}(q_1,q_2) R_{H_2|\bar{\alpha\beta}}=\left(zq_1\frac{d}{dq_1}+M_{\alpha\beta}\right)R(z)_{1|\bar{\alpha\beta}}\]
and
\[(1+I_{11}^{2}(q_2,q_1))R(z)_{H_2|\bar{\alpha\beta}}+I_{11}^{2}(q_2,q_1) R_{H_1|\bar{\alpha\beta}}=\left(zq_2\frac{d}{dq_2}+L_{\alpha\beta}\right)R(z)_{1|\bar{\alpha\beta}}\]
here
\[R(z)_{1|\bar{\alpha\beta}}=\|e_{\alpha\beta}\|{R(z)_{1}}^{\alpha\beta}
=\Psi_{1|\bar{\alpha\beta}}{R(z)_{1}}^{\alpha\beta}\]
with $\|e_{\alpha\beta}\|=\Psi_{1|\bar{\alpha\beta}}$ is the norm of the quantum canonical basis.
This is equivalent to
\[(1+I_{11}^{2}(q_1,q_2)){R(z)_{H_1}}^{\alpha\beta}+I_{11}^{2}(q_1,q_2) {R_{H_2}}^{\alpha\beta}=\|e_{\alpha\beta}\|^{-1}zq_1\frac{d}{dq_1}\left(
\|e_{\alpha\beta}\|{R(z)_{1}}^{\alpha\beta}\right)+M_{\alpha\beta}{R(z)_{1}}^{\alpha\beta}\]
and
\[(1+I_{11}^{2}(q_2,q_1)){R(z)_{H_2}}^{\alpha\beta}+I_{11}^{2}(q_2,q_1) {R_{H_1}}^{\alpha\beta}=\|e_{\alpha\beta}\|^{-1}zq_2\frac{d}{dq_2}\left(
\|e_{\alpha\beta}\|{R(z)_{1}}^{\alpha\beta}\right)+L_{\alpha\beta}{R(z)_{1}}^{\alpha\beta}\]
Solving these two linear equations, we get equations $(i)$ and $(ii)$ in Lemma~\ref{lem:Rrecursion}.


The second and third columns of quantum differential equations~\eqref{eq:qde-R}  give the recursion relations
\begin{footnotesize}
\[\left(\lambda^2 I_{22a}^{1;\lambda^2}(q_1,q_2))+\mu^2 I_{22a}^{1;\mu^2}(q_1,q_2)\right){R(z)_{1}}^{\alpha\beta}+I_{22}^{1}(q_1,q_2) {R_{H_1H_2}}^{\alpha\beta}=\|e_{\alpha\beta}\|^{-1}zq_1\frac{d}{dq_1}\left(\|e_{\alpha\beta}\|{R(z)_{H_1}}^{\alpha\beta}\right)+M_{\alpha\beta}{R(z)_{H_1}}^{\alpha\beta}\]
\end{footnotesize}
\begin{footnotesize}
\[\left(\lambda^2 I_{22a}^{2;\lambda^2}(q_1,q_2))+\mu^2 I_{22a}^{2;\mu^2}(q_1,q_2)\right){R(z)_{1}}^{\alpha\beta}+I_{22}^{2}(q_1,q_2) {R_{H_1H_2}}^{\alpha\beta}=\|e_{\alpha\beta}\|^{-1}zq_1\frac{d}{dq_1}\left(\|e_{\alpha\beta}\|{R(z)_{H_2}}^{\alpha\beta}\right)+M_{\alpha\beta}
{R(z)_{H_2}}^{\alpha\beta}\]
\end{footnotesize}
and
\begin{footnotesize}
\[\left(\mu^2 I_{22a}^{2;\lambda^2}(q_2,q_1))+\lambda^2 I_{22a}^{2;\mu^2}(q_2,q_1)\right){R(z)_{1}}^{\alpha\beta}+I_{22}^{2}(q_2,q_1) {R_{H_1H_2}}^{\alpha\beta}=\|e_{\alpha\beta}\|^{-1}zq_2\frac{d}{dq_2}\left(\|e_{\alpha\beta}\|{R(z)_{H_1}}^{\alpha\beta}\right)+L_{\alpha\beta}{R(z)_{H_1}}^{\alpha\beta}\]
\end{footnotesize}
\begin{footnotesize}
\[\left(\mu^2 I_{22a}^{1;\lambda^2}(q_2,q_1))+\lambda^2 I_{22a}^{1;\mu^2}(q_2,q_1)\right){R(z)_{1}}^{\alpha\beta}+I_{22}^{1}(q_2,q_1) {R_{H_1H_2}}^{\alpha\beta}=\|e_{\alpha\beta}\|^{-1}zq_2\frac{d}{dq_2}\left(\|e_{\alpha\beta}\|{R(z)_{H_2}}^{\alpha\beta}\right)+L_{\alpha\beta}
{R(z)_{H_2}}^{\alpha\beta}\]
\end{footnotesize}
Sum the above 4 equations, we get equation $(iii)$ in Lemma~\ref{lem:Rrecursion}.
\end{proof}
Then following from Lemma~\ref{lem:Rrecursion} combining with the Proposition~\ref{pro:Rstar1}, we get the corollary
\begin{corollary}\label{cor:R-ring}
For the $R$ matrix of the $(\lambda,\mu)-$ equivariant Gromov-Witten theory of $K_{\mathbb{P}^1\times\mathbb{P}^1}$,
\begin{small}
\begin{align*}
{(R_k)_{1}}^{\alpha\beta}
\in& \mathbb{G}_{3k,8k},
\\
{(R_k)_{H_1}}^{\alpha\beta}
\in& \mathbb{G}_{3k,8k+1}\otimes_{\mathbb{Q}}\mathbb{Q}\left\{\frac{1+I_{11}^{2}(q_2,q_1)}{1+\bar{I}_{11}^{2}(q_1,q_2)},\frac{I_{11}^{2}(q_1,q_2)}{1+\bar{I}_{11}^{2}(q_1,q_2)}\right\},
\\
{(R_k)_{H_2}}^{\alpha\beta}
\in& \mathbb{G}_{3k,8k+1}\otimes_{\mathbb{Q}}\mathbb{Q}\left\{\frac{1+I_{11}^{2}(q_1,q_2)}{1+\bar{I}_{11}^{2}(q_1,q_2)},\frac{I_{11}^{2}(q_2,q_1)}{1+\bar{I}_{11}^{2}(q_1,q_2)}\right\},
\\
{(R_k)_{H_1H_2}}^{\alpha\beta}
\in& \mathbb{G}_{3k,8k+2}\otimes_{\mathbb{Q}}\mathbb{Q}\Bigg\{\frac{X(q_1,q_2)}{(1+\bar{I}_{11}^{2}(q_1,q_2))(\tilde{I}_{22}(q_1,q_2)+\tilde{I}_{22}(q_2,q_1))}, \
\frac{ \tilde{I}_{22}(q_1,q_2)}{\tilde{I}_{22}(q_1,q_2)+\tilde{I}_{22}(q_2,q_1)},
\\&\hspace{70pt}\frac{1}{(1+\bar{I}_{11}^{2}(q_1,q_2))(\tilde{I}_{22}(q_1,q_2)+\tilde{I}_{22}(q_2,q_1))},\  \frac{ \tilde{I}_{22}(q_2,q_1)}{\tilde{I}_{22}(q_1,q_2)+\tilde{I}_{22}(q_2,q_1)}\Bigg\}.
\end{align*}
\end{small}
Moreover,
\begin{align*}
{(R_k)_{H_1}}^{\alpha\beta}
+{(R_k)_{H_2}}^{\alpha\beta}
\in& \mathbb{G}_{3k,8k+1}\otimes_{\mathbb{Q}}\mathbb{Q}\left\{\frac{1}{1+\bar{I}_{11}^{2}(q_1,q_2)}\right\},\\
\mu^2{(R_k)_{H_1}}^{\alpha\beta}
+\lambda^2{(R_k)_{H_2}}^{\alpha\beta}
\in &
\mathbb{G}_{3k,8k+3}\otimes_{\mathbb{Q}}\mathbb{Q}\left\{\frac{1}{1+\bar{I}_{11}^{2}(q_1,q_2)},\frac{I_{11}^{2}(q_1,q_2)}{1+\bar{I}_{11}^{2}(q_1,q_2)}\right\}.
\end{align*}
\end{corollary}
\begin{proof}
Following from  the recursion formula $(i),(ii)$ in the Lemma~\ref{lem:Rrecursion}, we obtain
\begin{align*}
{R(z)_{H_1}}^{\alpha\beta}
+{R(z)_{H_2}}^{\alpha\beta}
=&\frac{ L_{\alpha\beta}+M_{\alpha\beta}}{1+I_{11}^{2}(q_1,q_2)+I_{11}^{2}(q_2,q_1)}{R(z)_1}^{\alpha\beta}
\\&+z\frac{1}{1+I_{11}^{2}(q_1,q_2)+I_{11}^{2}(q_2,q_1)}\Bigg(q_2\frac{d}{dq_2}{R(z)_1}^{\alpha\beta}+q_1\frac{d}{dq_1}{R(z)_1}^{\alpha\beta}
\\&\hspace{150pt}
+\left(\frac{q_1\frac{d}{dq_1}\|e_{\alpha\beta}\|}{\|e_{\alpha\beta}\|}+\frac{q_2\frac{d}{dq_2}\|e_{\alpha\beta}\|}{\|e_{\alpha\beta}\|}\right){R(z)_1}^{\alpha\beta}\Bigg)
\end{align*}
and
\begin{align*}
&\mu^2{R(z)_{H_1}}^{\alpha\beta}
+\lambda^2{R(z)_{H_2}}^{\alpha\beta}
\\=&\frac{ \lambda^2 L_{\alpha\beta}+\mu^2 M_{\alpha\beta}+I_{11}^2(q_1,q_2)(\lambda^2-\mu^2)L_{\alpha\beta}+I_{11}^2(q_2,q_1)(\mu^2-\lambda^2)M_{\alpha\beta}}{1+I_{11}^{2}(q_1,q_2)+I_{11}^{2}(q_2,q_1)}{R(z)_1}^{\alpha\beta}
\\&+z\frac{1}{1+I_{11}^{2}(q_1,q_2)+I_{11}^{2}(q_2,q_1)}\Bigg(\lambda^2q_2\frac{d}{dq_2}{R(z)_1}^{\alpha\beta}+\mu^2q_1\frac{d}{dq_1}{R(z)_1}^{\alpha\beta}
\\&\hspace{150pt}+(\lambda^2-\mu^2)I_{11}^2(q_1,q_2)\left(\left(\frac{q_2\frac{d}{dq_2}\|e_{\alpha\beta}\|}{\|e_{\alpha\beta}\|}+q_2\frac{d}{dq_2}\right){R(z)_{1}}^{\alpha\beta}\right)
\\&\hspace{150pt}+(\mu^2-\lambda^2)I_{11}^2(q_2,q_1)\left(\left(\frac{q_1\frac{d}{dq_1}\|e_{\alpha\beta}\|}{\|e_{\alpha\beta}\|}+q_1\frac{d}{dq_1}\right){R(z)_{1}}^{\alpha\beta}\right)
\\&\hspace{150pt}
+\left(\mu^2\frac{q_1\frac{d}{dq_1}\|e_{\alpha\beta}\|}{\|e_{\alpha\beta}\|}+\lambda^2\frac{q_2\frac{d}{dq_2}\|e_{\alpha\beta}\|}{\|e_{\alpha\beta}\|}\right){R(z)_1}^{\alpha\beta}\Bigg)
\end{align*}
By the formula~\eqref{eq:norme}, we can compute
\[\frac{q_1\frac{d}{dq_1}\|e_{\alpha\beta}\|}{\|e_{\alpha\beta}\|},\ \frac{q_2\frac{d}{dq_2}\|e_{\alpha\beta}\|}{\|e_{\alpha\beta}\|}\in \mathbb{G}_{2,6}; \quad q_1\frac{d}{dq_1}M_{\alpha\beta},\ q_2\frac{d}{dq_2}M_{\alpha\beta}, \ q_1\frac{d}{dq_1}L_{\alpha\beta},\ q_2\frac{d}{dq_2}L_{\alpha\beta}\in \mathbb{G}_{1,4}.\]
Then together with induction on the behavior of $R_k$, we can get this corollary.
\end{proof}

\section{Finite generation property of $F_g$}\label{sec:finit-gen}
In this section, we prove the theorem of finite generation property for  genus $g$ generating function $F_g(\tau(q_1,q_2))$.
To state the theorem more precisely, we define some  auxiliary  generators
\[P_1=\frac{1}{1+\bar{I}_{11}^2},\
P_2=\frac{I_{11}^2(q_1,q_2)}{1+\bar{I}_{11}^2},\]
\[P_3=\frac{1}{\tilde{I}_{22}(q_1,q_2)+\tilde{I}_{22}(q_2,q_1)}, \ P_4=\frac{\tilde{I}_{22}(q_1,q_2)}{\tilde{I}_{22}(q_1,q_2)+\tilde{I}_{22}(q_2,q_1)}.\]

\begin{theorem}
\label{thm:finit-gen}For $g\geq2$, the genus-$g$ generating function of the equivariant Gromov-Witten invariants
\[F_g(\tau(q_1,q_2))
\in \mathbb{G}\left[P_1,P_2,P_3,P_4,X\right]\]
Moreover, the degree of $X$ in the polynomial expression of $F_g$ is at most $3g-3$.

\end{theorem}
\begin{proof}
By Givental-Teleman theorem (c.f. \cite{teleman2012structure}), the genus $g$ generating funciton $F_{g}(\tau(q_1,q_2))$ equals to a graph sum formula  (c.f. \cite{pandharipande2015relations})
\[F_g(\tau(q_1,q_2))=\sum_{\Gamma\in G_{g,0}}\text{Cont}_{\Gamma}F_g\]
where $G_{g,0}$ is the set of all genus $g$, $0$ marking stable graphs. Below we will show
\[\text{Cont}_{\Gamma}F_g\in \mathbb{G}[P_1,P_2,P_3,P_4,X]\] and the degree of $X$ in the polynomial of $F_g$ is at most $|E(\Gamma)|$.

For general genus-$g$, $0$ marking  stable graph $\Gamma$, the associated contribution $\text{Cont}_{\Gamma}F_g$
 in Givental-Teleman graph sum formula  is
\begin{itemize}
\item{On each vertex, the contribution is
\[\omega_{g_v,n_v}(e_{\alpha\beta},...,e_{\alpha\beta})={\Delta_{\alpha\beta}}^{\frac{2g_v-2}{2}}
\]}
\item{On the kappa tails, the contribution is
\[z({\Psi_1}^{\alpha\beta}-{R(-z)_1}^{\alpha\beta})=z\cdot \sum_{k\geq1} T_k\]
where $T_k$ is a polynomial of $z$ with degree $k$.}
\item{On the edges, the contribution is
\begin{align*}
\langle V(z,w), e^{\alpha\beta}\otimes e^{\gamma\delta}\rangle=V_0+V_1+V_2+...
\end{align*}
where
\[V(z,w)=\sum_{\alpha,\beta}\frac{e_{\alpha\beta}\otimes e^{\alpha\beta}-R(z)^{-1}e_{\alpha\beta}\otimes R(w)^{-1} e^{\alpha\beta}}{z+w} \]
and $V_k$ is a homogeneous polynomial of $z,w$ with degree $k$.
}
\end{itemize}
Then by equation~\eqref{eq:norme} and Corollary~\ref{cor:R-ring}, we have
\begin{align*}
&\omega_{g_v,n_v}(e_{\alpha\beta},...,e_{\alpha\beta})
=(-2)^{3g_v-3}\left(\lambda^2L_{\alpha\beta}+\mu^2M_{\alpha\beta}\right)^{g_v-1}\in \mathbb{G}_{1,3g_v}
\\
&T_k \in \mathbb{G}_{3k,8k}
\end{align*}
and
\begin{align*}V_k\in&
\mathbb{Q}[z,w]_{\deg=k}\otimes_{\mathbb{Q}}\mathbb{G}_{3(k+1),8(k+1)+3}\otimes_{\mathbb{Q}}
\\&\mathbb{Q}\Bigg\{\frac{1}{1+\bar{I}_{11}^2},\frac{I_{11}^2(q_1,q_2)}{1+\bar{I}_{11}^2},\frac{1}{(1+\bar{I}_{11})^2(\tilde{I}_{22}(q_1,q_2)+\tilde{I}_{22}(q_2,q_1))},
\\&\hspace{15pt}\frac{\tilde{I}_{22}(q_1,q_2)}{(1+\bar{I}_{11})(\tilde{I}_{22}(q_1,q_2)+\tilde{I}_{22}(q_2,q_1))},\frac{\tilde{I}_{22}(q_2,q_1)}{(1+\bar{I}_{11})(\tilde{I}_{22}(q_1,q_2)+\tilde{I}_{22}(q_2,q_1))},\\
&\hspace{15pt}\frac{X}{(1+\bar{I}_{11})^2(\tilde{I}_{22}(q_1,q_2)+\tilde{I}_{22}(q_2,q_1))}\Bigg\}.
\end{align*}

Assume $V(\Gamma), E(\Gamma),T(\Gamma)$ are the set of vertices, edges, tails of the stable graph $\Gamma$.
\[\text{Cont}_{\Gamma}F_g=\frac{1}{|\text{Aut}(\Gamma)|}
\sum_{k_1,...,k_{|E(\Gamma)|}}\sum_{l_1,...,l_{|T(\Gamma)|}}c_{k_1,...,k_{|E(\Gamma)|};l_1,...,l_{|E(\Gamma)|}}\left(\prod_{v\in V(\gamma)}\omega_{g_v,n_v}\right)\left(\prod_{i=1}^{|E(\Gamma)|}V_{k_i}\right)\left(\prod_{j=1}^{|T(\Gamma)|}T_{l_j}\right)\]
where the constant $c_{k_1,...,k_{|E(\Gamma)|};l_1,...,l_{|E(\Gamma)|}}$ comes from  the integration of product of corresponding $\psi$ classes over $\overline{\mathcal{M}}_{g}$.
To make sure $c_{k_1,...,k_{|E(\Gamma)|};l_1,...,l_{|E(\Gamma)|}}$ is nonzero, we need requirements \begin{align}\label{eq:nonreq}\sum_{i=1}^{|E(\Gamma)|}{k_i}+\sum_{j=1}^{|T(\Gamma)|}l_j=3g-3-|E(\Gamma)|.\end{align}
Then the product of these factors in $\mathbb{G}_{\bullet,\bullet}$ gives a factor lies in $\mathbb{G}$. In fact,
under the requirement \eqref{eq:nonreq}, we have \[3\left(\sum_{v\in V(\Gamma)}1+\sum_{j=1}^{|T(\Gamma)|}3l_j+\sum_{i=1}^{|E(\Gamma)|}3(k_i+1)\right)
=\sum_{v\in V(\Gamma)}3g_v+\sum_{j=1}^{|T(\Gamma)|}8l_j+\sum_{i=1}^{E(\Gamma)}\left(8(k_i+1)+3\right).\]
The remains gives a product lies in the ring
\[ \mathbb{Q}[P_1,P_2,P_3,P_4,X]\] and the degree of $X$ in the polynomial of $F_g$ is at most $|E(\Gamma)|$.

Since for any stable graph $\Gamma$ of genus $g$, 0 marking, the number of edges $|E(\Gamma)|$ is at most 3g-3, we prove this theorem.
\end{proof}

\section{Holomorphic anomaly equation }
\label{sec:HAE}
In this section, we prove the holomorphic anomaly equation for the equivariant  Gromov-Witten theory of local $\mathbb{P}^1\times\mathbb{P}^1$.
\subsection{Derivatives of the full $R$ matrix}
The key point for the proof is the following proposition
\begin{proposition}
Formally we have the following derivatives of full $R$ and bi-vector $V$ w.r.t. $X$
\begin{align*}
\frac{d}{dX}\left({R(z)_{1}}^{\alpha\beta}\right)
=&0,
\\ \frac{d}{dX}\left({R(z)_{H_1}}^{\alpha\beta}
\right)=&0,
\\ \frac{d}{dX}\left({R(z)_{H_2}}^{\alpha\beta}\right)=&0,
\\  \frac{d}{dX}\left({R(z)_{H_1H_2}}^{\alpha\beta}
\right)=&-z\frac{1}{\tilde{I}_{22}(q_1,q_2)+\tilde{I}_{22}(q_2,q_1)}\left(
{R(z)_{H_1}}^{\alpha\beta}+{R(z)_{H_2}}^{\alpha\beta}\right),
\\
\frac{d}{dX}\langle V(z,w),e^{\alpha\beta}\otimes e^{\gamma\delta}\rangle
=&
-\frac{1}{\tilde{I}_{22}(q_1,q_2)+\tilde{I}_{22}(q_2,q_1)}\langle R(z)^{-1}(H_1+H_2),e^{\alpha\beta}\rangle \langle R(w)^{-1}(H_1+H_2), e^{\gamma\delta}\rangle.
\end{align*}
\end{proposition}
\begin{proof}
The first four equations about the derivatives of the $R$ matrix along $X$ just follows from the Lemma~\ref{lem:Rrecursion} and Corollary~\ref{cor:R-ring}.

To see the last equation, we need to express the bivector $V(z,w)$ in terms of $R$ matrix.
Via the paring matrix~\eqref{paring-matrix}, we get the dual basis of $\{1,H_1,H_2,H_1H_2\}$ is
\[-2\{\mu^2H_1+\lambda^2 H_2, H_1H_2+\mu^2, H_1H_2+\lambda^2, H_1+H_2\}.\]
So the bivector $V(z,w)$  equals to the following
\begin{align*}
V(z,w)=\frac{-2}{z+w}
\Bigg(&\bigg(1\otimes(\mu^2H_1+\lambda^2H_2)+(\mu^2H_1+\lambda^2H_2)\otimes 1
\\&+(H_1+H_2)\otimes H_1H_2+H_1H_2\otimes (H_1+H_2)\bigg)
\\&-\bigg(R(z)^{-1}1\otimes R(w)^{-1}(\mu^2 H_1+\lambda^2 H_2)+R(z)^{-1}(\mu^2 H_1+\lambda^2 H_2)\otimes R(w)^{-1}1
\\&+R(z)^{-1}(H_1+H_2)\otimes R(w)^{-1}(H_1H_2)+R(z)^{-1}(H_1H_2)\otimes R(w)^{-1}(H_1+H_2)\bigg)\Bigg)
\end{align*}
Then the identity of the derivative of $V$ w.r.t. $X$ follows from equations of the derivative of $R$ w.r.t X.
\end{proof}

\subsection{Finite generation property}
\begin{theorem}\label{thm:finit-gene}
The equivariant Gromov-Witten invariants of $K_{\mathbb{P}^1\times\mathbb{P}^1}$ satisfies the holomorphic anomaly equation
\begin{align}\label{eq:HAE-thm}
&\frac{d}{dX}F_g(\tau(q_1,q_2))
\\=&-\frac{1}{2\left(\tilde{I}_{22}(q_1,q_2)+\tilde{I}_{22}(q_2,q_1)\right)}
\left(\sum_{g_1+g_2=g}\langle\langle
H_1+H_2\rangle\rangle_{g_1,1}\langle\langle
H_1+H_2\rangle\rangle_{g_2,1}
+\langle\langle
H_1+H_2,H_1+H_2\rangle\rangle_{g-1,2}\right).
\nonumber\end{align}
\end{theorem}
\begin{proof}
In the proof of Theorem~\ref{thm:finit-gen}, we have proved that,  by Givental-Teleman theorem, the genus $g$ generating funciton $F_{g}(\tau(q_1,q_2))$ equals to a graph sum formula
\[F_g(\tau(q_1,q_2))=\sum_{\Gamma\in G_{g,0}}\text{Cont}_{\Gamma}F_g\]
here
\[\text{Cont}_{\Gamma}F_g\in \mathbb{G}[P_1,P_2,P_3,P_4,X]\] and the degree of $X$ in the polynomial of $F_g$ is at most $|E(\Gamma)|$.

Recall that  the associated contribution $\text{Cont}_{\Gamma}F_g$
 in Givental-Teleman graph sum formula  is
\begin{itemize}
\item{On each vertex, the contribution is
\[\omega_{g_v,n_v}(e_{\alpha\beta},...,e_{\alpha\beta})={\Delta_{\alpha\beta}}^{\frac{2g_v-2}{2}}
\]}
\item{On the kappa tails, the contribution is
\[z({\Psi_1}^{\alpha\beta}-{R(-z)_1}^{\alpha\beta})\]
}
\item{On the edges, the contribution is
\begin{align*}
\langle V(z,w), e^{\alpha\beta}\otimes e^{\gamma\delta}\rangle.
\end{align*}
}
\end{itemize}

As is in the proof of theorem~\ref{thm:finit-gen},
\[\text{Cont}_{\Gamma}F_g=\frac{1}{|\text{Aut}(\Gamma)|}
\sum_{k_1,...,k_{|E(\Gamma)|}}\sum_{l_1,...,l_{|T(\Gamma)|}}c_{k_1,...,k_{|E(\Gamma)|};l_1,...,l_{|E(\Gamma)|}}\left(\prod_{v\in V(\gamma)}\omega_{g_v,n_v}\right)\left(\prod_{i=1}^{|E(\Gamma)|}V_{k_i}\right)\left(\prod_{j=1}^{|T(\Gamma)|}T_{l_j}\right).\]
Taking derivatives along the generator $X$, by Lebniz's rule,  we get
\begin{align*}
&\frac{\partial}{\partial X}\text{Cont}_{\Gamma}F_g(\tau(q_1,q_2))
\\=&-\frac{1}{2\left(\tilde{I}_{22}(q_1,q_2)+\tilde{I}_{22}(q_2,q_1)\right)}\Bigg(\sum_{\Gamma'\in G_{g-1,2}^{\Gamma}}\text{Cont}_{\Gamma}\langle\langle H_1+H_2,H_1+H_2\rangle\rangle_{g-1,2}
\\&\hspace{140pt}+\sum_{(\Gamma'',\Gamma''')\in (G_{g_1,1}\times G_{g_2,1})^{\Gamma}}\langle\langle H_1+H_2\rangle\rangle_{g_1,1}\langle\langle H_1+H_2\rangle\rangle_{g_2,1}\Bigg).
\end{align*}
where $G_{g-1,2}^{\Gamma}$ is the set of stable graphs with genus $g$, 2 markings, which becomes the stable graph $\Gamma$ under the gluing map
$\iota_1: \overline{M}_{g-1,2}\rightarrow \overline{M}_{g}$.  $(G_{g_1,1}\times G_{g_2,1})^{\Gamma}$ is the set of two stable graphs with genus $g_1$, 1 marking and genus $g_2$, 1 marking, which becomes the stable graph $\Gamma$ under the gluing map
$\iota_2: \overline{M}_{g_1,1}\times \overline{M}_{g_2,1}\rightarrow \overline{M}_{g}$. Then summing over all stable graph of genus $g$, 0 marking, we get the holomorphic anomaly equation~\eqref{eq:HAE-thm}.
\end{proof}

\section{Oscillatory integral and Feymann diagram  representation of $R$ matrix }\label{sec:osc-int}
In the last section, we
study the associated oscillatory integral of $K_{\mathbb{P}^1\times\mathbb{P}^1}$ and give a Feymann diagram representation of the $R$ matrix, which provides a directly combinatorial way to compute the $R$ matrix.

Recall that the charge matrix of the toric variety $K_{\mathbb{P}^1\times\mathbb{P}^1}$ is
\[\begin{pmatrix}
1&1&0&0&-2\\
0&0&1&1&-2
\end{pmatrix}\]
The LG potential is defined by
\[W=\sum_{i=0}^{4}(x_i+\lambda_i\ln x_i): (\mathbb{C}^*)^5\rightarrow \mathbb{C}\]
In the following, we choose the  specialization of equivariant parameter $\lambda_0=-\lambda_1=\lambda$ and $\lambda_2=-\lambda_3=\mu$, $\lambda_4=0$.

The associated oscillatory integral is of the form
\[\int_{\Gamma}\exp(\frac{W}{z})\frac{d\ln x_0...d\ln x_4}{d\ln q_1 d\ln q_2}\]
where $\Gamma$ is a subset in as sub-torus in $(\mathbb{C}^*)^5$, the sub-torus is given by the equations
\[\left\{\begin{aligned}
x_0x_1{x_4}^{-2}&=&q_1
\\x_2x_3{x_4}^{-2}&=&q_2.
\end{aligned}
\right.\]
Then the oscillatory integral becomes
\begin{align*}
\frac{1}{2}\int_{\Gamma}\exp(\frac{W}{z})d\ln x_0d\ln x_1d\ln x_2
\end{align*}
where
\begin{align*}
W|_{\Gamma}=x_0+x_1+x_2+\frac{q_2}{q_1}\frac{x_0x_1}{x_2}
+\left(\frac{x_0x_1}{q_1}\right)^{\frac{1}{2}}
+\lambda_0 \ln x_0+\lambda_1\ln x_1+\lambda_2\ln x_2+\lambda_3\ln (\frac{q_2}{q_1}\frac{x_0x_1}{x_2}).
\end{align*}
The critical points of $W|_{\Gamma}$ is determined by the equation
\[ \left\{
\begin{aligned}
x_0+\lambda_0+x_3+\lambda_3+\frac{1}{2}x_4 & = & 0 \\
x_1+\lambda_1+x_3+\lambda_3+\frac{1}{2}x_4 & = & 0 \\
x_2+\lambda_2-(x_3+\lambda_3) & = & 0\\
x_0x_1{x_4}^{-2}&=&q_1\\
x_2x_3{x_4}^{-2}&=&q_2
\end{aligned}
\right.
\]Now we can assume $x_2+\lambda_2=x_3+\lambda_3=\check{L}$, $x_0+\lambda_0=x_1+\lambda_1=\check{M}$, then $x_4=-2(\check{M}+\check{L})$, where $(\check{M},\check{L})$ is solutions of the equations
\[\left\{
\begin{aligned}
{\check{M}}^2-\lambda^2& = & q_1(2(\check{M}+\check{L}))^2 \\
{\check{L}}^2-\mu^2& = & q_2(2(\check{M}+\check{L}))^2
\end{aligned}
\right.\]
Note that
 for general $\lambda$ and  $\mu$, there are 4 different solutions $\big\{(M_{\alpha\beta},L_{\alpha\beta})|\alpha,\beta\in\{0,1\}\big\}$.
\begin{definition}
A {\it Feymann diagram with $m$ external vertices} is a graph
(possibly, with loops and multiple edges) with $m$ vertices of
degree 1 labeled by $1, 2, . . . , m$ and finitely many unlabeled
(internal) vertices of degrees $\geq 3$.
\end{definition}
Let $G^{\geq3,\bullet}(l)$ denotes isomorphism classes of  the possibly disconnected Feymann diagram with $l$ external vertices. For any Feymann diagram $\Gamma\in G^{\geq3,\bullet}(l)$,  $\text{Aut}(\Gamma)$  denotes number of  automorphisms of  $\ \Gamma$ leaving the external vertices fixed.
\begin{proposition}
\label{prop:R-Fey}
The first column of $R$ matrix has the following Feymann diagram representation
\begin{align}\label{eq:R-Feymann}
{R(z)_1}^{\alpha\beta}=1+\sum_{k\geq1}\sum_{l\geq0}(-z)^{k}\sum_{\emptyset\neq\Gamma\in G_{k}^{\geq3,\bullet}(l)}\frac{1}{|\text{Aut}(\Gamma)|}\text{Cont}_{\Gamma}
\left(\frac{1}{ -x^{(\alpha\beta)}_0},\frac{1}{ -x^{(\alpha\beta)}_1},\frac{1}{ -x^{(\alpha\beta)}_2}\right)
\end{align}
where
$$G_{k}^{\geq3,\bullet}(l)
=\left\{\Gamma\in G^{\geq3,\bullet}(l): |E(\Gamma)|-|V_{\text{internal}}|=k\right\}$$
is the subset of $G^{\geq3,\bullet}(l)$  whose number of edges minus the number of internal vertices  equals to $k$.
The contribution of each Feymann diagram $\text{Cont}_{\Gamma}
\left(\frac{1}{ -x^{(\alpha\beta)}_0},\frac{1}{ -x^{(\alpha\beta)}_1},\frac{1}{ -x^{(\alpha\beta)}_2}\right)$ is given by
\begin{itemize}
    \item{On the $j$-th external vertex, we place the linear form $\frac{\xi_{i_j}}{-x^{(\alpha\beta)}_{i_j}}$, where $\frac{\xi_{i_j}}{-x^{(\alpha\beta)}_{i_j}}$ is the $j$-th factor of degree $l$ monomial in the multi-taylor expansion $\frac{1}{\left(1+\frac{\xi_0}{x_0^{(\alpha\beta)}}\right)\left(1+\frac{\xi_1}{x_1^{(\alpha\beta)}}\right)\left(1+\frac{\xi_2}{x_2^{(\alpha\beta)}}\right)}$. }
    \item{On each $m$ valent internal vertex, we place polylinear form $-W_m$.}
    \item{On the edges, we take the contraction of these forms by $\text{Hess}(W)^{-1}$ along the each edge.}
    \end{itemize}
\end{proposition}
\begin{proof}
Firstly,  using saddle point method, we compute the oscillatory integral
\begin{small}\begin{align}
\label{eqn:osi-int}
&\frac{1}{2}\int_{\Gamma}\exp(\frac{W}{z})d\ln x_0d\ln x_1d\ln x_2
\\=&\frac{1}{2}e^{\frac{\check{u}_{\alpha\beta}}{z}}\int_{\Gamma_{\alpha\beta}}
\exp\left(\frac{1}{z}\left(\frac{1}{2!}\sum_{i,j=0}^{2}
\partial_{i}\partial_{j}W|_{x^{(\alpha\beta)}}\xi_i\xi_j
+\frac{1}{k!}\sum_{i_1,...,i_k=0}^{2}\partial_{i_1}...\partial_{i_k}W|_{x^{(\alpha\beta)}}
\xi_{i_1}...\xi_{i_k}\right)\right)\frac{d\xi_0d\xi_1d\xi_2}{
(\xi_0+x_0^{(\alpha\beta)})(\xi_1+x_1^{(\alpha\beta)})(\xi_2+x_2^{(\alpha\beta)})}\nonumber
\\=&(\sqrt{-2\pi z})^3e^{\frac{\check{u}_{\alpha\beta}}{z}}
\left(\frac{1}{2}\frac{1}{x^{(\alpha\beta)}_0x^{(\alpha\beta)}_1
 x^{(\alpha\beta)}_2}
 \frac{1}{\sqrt{\text{Hess}(\partial^2W)}}\right)
 \frac{\sqrt{\text{Hess}(\partial^2W)}}{\sqrt{(2\pi)^3}}\nonumber
 \\&\int_{\tilde{\Gamma_{\alpha\beta}}
}
e^{-\frac{1}{2}\sum_{i,j=0}^{2}\partial_{i}\partial_{j}W(\{x^{(\alpha\beta)}\})\tilde{\xi}_i\tilde{\xi}_j}
e^{-\sum_{k\geq3}\sum_{I=(i_1,...,i_k)}
\frac{(-z)^{\frac{k}{2}-1}}{k!}\partial_{I}W(\{x^{(\alpha\beta)}\})\tilde{\xi}_{I}}
\frac{
d\tilde{\xi}_0\wedge d\tilde{\xi}_1\wedge d\tilde{\xi}_2}{(1+\sqrt{-z}\frac{\tilde{\xi}_0}{ x^{(\alpha\beta)}_0})(1+\sqrt{-z}\frac{\tilde{\xi}_1}{ x^{(\alpha\beta)}_1})
(1+\sqrt{-z}\frac{\tilde{\xi}_2}{x^{(\alpha\beta)}_2})}\nonumber
\\=&(\sqrt{-2\pi z})^3e^{\frac{\check{u}_{\alpha\beta}}{z}}\Psi_{1;\bar{\alpha\beta}}{R(z)_1}^{\alpha\beta}.\nonumber
\end{align}
\end{small}
Here the last step just  follows from the well known relations between $R$ matrix and asymptotic expansions of oscillatory integral of smooth toric variety (c.f. \cite{givental2001gromov}).
By the oscillatory integral \eqref{eqn:osi-int},  the norm of quantum canonical basis: \begin{align}\label{eq:norme}
{\Delta_{\alpha\beta}}^{-\frac{1}{2}}
=\Psi_{1;\bar{(\alpha,\beta)}}
=\frac{1}{2}\frac{1}{x_0^{(\alpha\beta)}x_1^{(\alpha\beta)}x_2^{(\alpha\beta)}}\frac{1}{\sqrt{\det\left(\text{Hess}(\partial^2W)\right)}}
=\frac{1}{2}\frac{1}{\sqrt{-2L_{\alpha\beta}\lambda^2-2M_{\alpha\beta}\mu^2}}=\|e_{\alpha\beta}\|
.\end{align}

Then by Wick's theorem (c.f. \cite{Etingof02}), , we have the following Feymann diagram expansion
\begin{small}\begin{align*}
&{R(z)_1}^{\alpha\beta}
\\=&\frac{\sqrt{\text{Hess}(\partial^2W)}}{\sqrt{(2\pi)^3}}\sum_{n_0,n_1,n_2\geq0}(\sqrt{-z})^{n_0+n_1+n_2}
\\&\int_{{\mathbb{R}}^3}
\left(\frac{\tilde{\xi}_0}{-x^{(\alpha\beta)}_0}\right)^{n_0}
\left(\frac{\tilde{\xi}_1}{-x^{(\alpha\beta)}_1}\right)^{n_1}
\left(\frac{\tilde{\xi}_2}{-x^{(\alpha\beta)}_2}\right)^{n_2}
e^{-\frac{1}{2}\sum_{i,j=0}^{2}\partial_{i}\partial_{j}W(\{x^{(\alpha\beta)}\})\tilde{\xi}_i\tilde{\xi}_j}
e^{-\sum_{k\geq3}\sum_{I=(i_1,...,i_k)}
\frac{(-z)^{\frac{k}{2}-1}}{k!}\partial_{I}W(\{x^{(\alpha\beta)}\})\tilde{\xi}_{I}}
d\tilde{\xi}_0\wedge d\tilde{\xi}_1\wedge d\tilde{\xi}_2
\\=&1+\sum_{k\geq1}\sum_{l\geq0}(-z)^{g}\sum_{\emptyset\neq\Gamma\in G_{k}^{\geq3,\bullet}(l)}\frac{1}{|\text{Aut}(\Gamma)|}\text{Cont}_{\Gamma}
\left(\frac{1}{ -x^{(\alpha\beta)}_0},\frac{1}{ -x^{(\alpha\beta)}_1},\frac{1}{ -x^{(\alpha\beta)}_2}\right).
\end{align*}
\end{small}
Here the contribution $\text{Cont}_{\Gamma}
\left(\frac{1}{ -x^{(\alpha\beta)}_0},\frac{1}{ -x^{(\alpha\beta)}_1},\frac{1}{ -x^{(\alpha\beta)}_2}\right)$ is exactly as stated in Proposition~\ref{prop:R-Fey}.
\end{proof}
\begin{remark}
For the other columns of $R$ matrix, we can use quantum differential equation to  get their Feymann diagram representations like equation~\eqref{eq:R-Feymann}.
\end{remark}
As an example, we can  compute the first column of $R_1$ matrix directly.
\begin{example}In the expression of ${(R_1)_1}^{\alpha\beta}$, there are  5 Feymann diagrams as follows:

\includegraphics[width=5.5in]{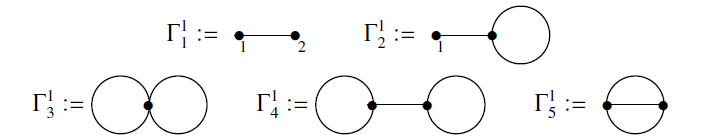}
\\Direct computations show for any $\alpha,\beta\in\{0,1\}$
\begin{small}\begin{align*}
&{(R_{1})_{1}}^{\alpha\beta}(q_1,q_2,\lambda,\mu)
\\=&\frac{1}{48(L_{\alpha\beta}\lambda^2+M_{\alpha\beta}\mu^2)^3}
\Bigg(9(L_{\alpha\beta}
\lambda^2+M_{\alpha\beta}\mu^2)\left(
\frac{1}{3}\lambda^4 L_{\alpha\beta}+\frac{1}{3} M_{\alpha\beta}\mu^4+L_{\alpha\beta}^2 M_{\alpha\beta}\lambda^2+L_{\alpha\beta} M_{\alpha\beta}^2\mu^2-2L_{\alpha\beta}\lambda^2\mu^2-2M_{\alpha\beta}\lambda^2\mu^2\right)
\\&\hspace{100pt}-5\lambda^2\mu^2({\lambda}^2-\mu^2)^2\Bigg).
\end{align*}
\end{small}
\end{example}


In general, we have the following proposition
\begin{proposition}\label{pro:Rstar1}
For $k\geq1$,  we have
\[{(R_k)_{1}}^{\alpha,\beta}(q_1,q_2,\lambda,\mu)
\in \frac{1}{(\lambda^2 L_{\alpha\beta}+\mu^2 M_{\alpha\beta})^{3k}}\cdot \mathbb{Q}\left[M_{\alpha\beta},L_{\alpha\beta},\lambda,\mu\right]_{\text{deg}=8k}\]
moreover,
for each Feymann diagram $\Gamma\in G_{k,l}^{\geq3}$, we have the contribution of $\Gamma$ in equation~\eqref{eq:R-Feymann},
\[\text{Cont}_{\Gamma}
\left(\frac{1}{ -x^{(\alpha\beta)}_0},\frac{1}{ -x^{(\alpha\beta)}_1},\frac{1}{ -x^{(\alpha\beta)}_2}\right)\in \frac{1}{(\lambda^2 L_{\alpha\beta}+\mu^2 M_{\alpha\beta})^{3k}}\cdot \mathbb{Q}\left[M_{\alpha\beta},L_{\alpha\beta},\lambda,\mu\right]_{\text{deg}=8k}\]

\end{proposition}
\begin{proof}
For any Feymann diagram $\Gamma\in G_{k,l}^{\geq3}$,
direct computations show that for any $i,j\in\{0,1,2\}$,
\[\frac{1}{x_i\cdot x_j}\cdot{\text{Hess}(W)^{-1}}_{ij}|_{x^{(\alpha\beta)}}
\in \frac{\mathbb{Q}[\lambda,\mu,x_3^{(\alpha\beta)},x_4^{(\alpha\beta)}]_{\hom.\deg=2}}{\lambda^2 L_{\alpha\beta}+\mu^2 M_{\alpha\beta}}.
\]
By induction, we can obtain that for any $i_1,...,i_m\in \{0,1,2\}$, $m\geq3$
\[x_{i_1}...x_{i_m}\partial_{i_1}...\partial_{i_k}W|_{x^{(\alpha\beta)}}
\in \mathbb{Q}\{\lambda,\mu,x_3^{(\alpha\beta)},x_4^{(\alpha\beta)}\}.\]
So formally in the contribution $\text{Cont}_{\Gamma}\left(-\frac{1}{x_0^{(\alpha\beta)}},-\frac{1}{x_1^{(\alpha\beta)}},-\frac{1}{x_2^{(\alpha\beta)}}\right)$, the power of $\frac{1}{\lambda^2 L_{\alpha\beta}+\mu^2 M_{\alpha\beta}}$
is the number of edges $|E(\Gamma)|$. The remains contribute a homogenous polynomial of $\lambda,\mu, M_{\alpha\beta},L_{\alpha\beta}$ with degree $2\cdot|E(\Gamma)|+|V_{\text{intermal}}|=3\cdot|E(\Gamma)|-k$. Moreover, for $\Gamma\in G_{k,l}^{\geq3}$, the number of edges $|E(\Gamma)|$ is at most $3k$.
\end{proof}

\begin{remark}
From the expression of ${R_{1}}^{\alpha\beta}$, we see under the specialization of equivariant parameter $\lambda=0$ or $\mu=0$, the expression becomes  much simpler. Actually, it is not hard to prove  for any $k\geq1$
\[{(R_k)_{1}}^{\alpha,\beta}(q_1,q_2,\lambda,\mu)|_{\mu=0}
\in \frac{1}{(\lambda^2 L_{\alpha\beta})^{k}}\cdot \mathbb{Q}\left[M_{\alpha\beta},L_{\alpha\beta},\lambda^2\right]_{\text{deg}=2k}.\]
Under this specialization of equivariant parameter $\mu=0$, we can prove the contribution of each stable graph in theorem~\ref{thm:finit-gen} is independent of $\lambda$. Thus the total sum of  the non-equivariant limit of each stable graph contribution equals to the Gromov-Witten potential of local $\mathbb{P}^1\times\mathbb{P}^1$. The drawbacks here is in general $Cont_{\Gamma}\in \mathbb{Q}(q_1,\sqrt{q_2})$ not $\mathbb{Q}(q_1,q_2)$. Moreover${(R_k)_{1}}^{\alpha,\beta}(q_1,q_2,\lambda,\mu)|_{\mu=0}$ is not defined along $q_2=0$ since the appearance of the pole in the denominator.
\end{remark}
\begin{remark}
Another interesting specialization of equivariant parameter is  $\lambda=\mu$ or $\lambda=-\mu$, and in this case,  the $R$ matrix becomes very simple. But ${(R_k)_1}^{\alpha,\beta}|_{\mu=\lambda}$ is not defined for $\alpha\neq\beta$, since now $L_{\alpha\beta}|_{\mu=\lambda}+M_{\alpha\beta}|_{\mu=\lambda}=0$ giving pole on the denominator.
\end{remark}
\begin{remark}
In paper \cite{lho2019gromov}, the author chose another specialization of equivariant parameter $\mu=\sqrt{-1}\cdot\lambda$.  Under this specialization, our $R$ matrix still very complicated. The good things in this case is $R$ matrix is regular at $(q_1,q_2)=(0,0)$.
\end{remark}
\bibliographystyle{abbrv}

\end{document}